\documentclass{amsart}
\usepackage{graphicx}

\usepackage[margin=30truemm]{geometry}

\usepackage{amsthm}
\newtheorem{theorem}{Theorem}[section]
\newtheorem{lemma}[theorem]{Lemma}
\newtheorem{proposition}[theorem]{Proposition}
\newtheorem{corollary}[theorem]{Corollary}

\newtheorem{conjecture}[theorem]{Conjecture}

\theoremstyle{definition}
\newtheorem{definition}[theorem]{Definition}
\newtheorem{example}[theorem]{Example}
\newcommand{\LeftEqNo}{\let\veqno\@@leqno}
\numberwithin{equation}{section}

\usepackage{amsmath,amsfonts} % AMSのたくさんの数学用の文字
\usepackage{mathrsfs, txfonts} % 筆記体と花文字

\newcommand{\C}{\mathbb{C}}
\newcommand{\R}{\mathbb{R}}
\newcommand{\Z}{\mathbb{Z}}

\newcommand{\RP}{\mathbb{R}\mathrm{P}}

\newcommand{\id}{\mathop{\mathrm{id}}\nolimits}

\newcommand{\Int}{\mathrm{Int}\,}

\begin{document}
\title{Characterizations of knot groups and knot symmetric quandles of surface-links}
\author{Jumpei Yasuda}
\address{Department of Mathematics, Graduate School of Science, Osaka University,  Toyonaka, Osaka 560-0043, Japan}
% \curraddr{}
\email{u444951d@ecs.osaka-u.ac.jp}
\keywords{Surface-link, Knot group, Symmetric quandle}
\subjclass[2020]{Primary 57K45, Secondary 57K10}
% \date{\today}

\begin{abstract} % *
    The knot group is the fundamental group of a knot or link complement.
    A necessary and sufficient conditions for a group to be realized as the knot group of some link was provided.
    This result was shown using the closed braid method.
    González-Acuña and Kamada independently extended this characterization to the knot groups of orientable surface-links.
    Kamada applied the closed 2-dimensional braid method to show this result.

    In this paper, we generalize these results to characterize the knot groups of surface-links, including non-orientable ones.
    We use a plat presentation for surface-links to prove it.
    Furthermore, we show a similar characterization for the knot symmetric quandles of surface-links.
    As an application, we show that every dihedral quandle with an arbitrarily good involution can be realized as the knot symmetric quandle of a surface-link.
\end{abstract}

\maketitle

\section{Introduction}
The \textit{knot group} is the fundamental group of the complement of a knot or a link.
In knot theory, the knot group plays a crucial role in the classification and study of invariants for knots and links.
This naturally leads to the following question:
for a given group $G$, when is $G$ the knot group of some link?
The answer to this question is provided by the following theorem:

\begin{theorem}[\cite{Alexander1923,Artin1925}]
    A group $G$ is the knot group of a link if and only if there exists an $m$-braid $b$ for some $m \geq 1$ such that $G$ has a presentation
    \[
        \left\langle x_1, \ldots, x_{m}~
        \begin{array}{|c}
           x_i = b \cdot x_i \quad (i=1,\ldots,m)
        \end{array}
      \right\rangle.
    \]
\end{theorem}

Here, $b\cdot x$ denotes the action on the free group $F_m$ by the braid group $B_m$, arising naturally from the work of Artin \cite{Artin1925}.
We will define this action in Section~\ref{Subsection: Braid systems of braided surfaces}.
This theorem follows from Alexander's theorem \cite{Alexander1923}, which states that every link is equivalent to the closure of some braid.

A \textit{surface-knot} is a closed connected surface smoothly embedded in the 4-space, and a \textit{surface-link} is a disjoint union of surface-knots.
Gonz\'{a}lez-Acu\~{n}a \cite{Gonzalez-Acuna94} and Kamada \cite{Kamada1994-01} independently proved a characterization of the knot groups of orientable surface-links.
We recall the result due to Kamada as follows:
% Let $b_1, \dots, b_n$ be $m$-braids.
A group presentation is called an \textit{$(m,n)$-presentation} (\textit{associated with $m$-braids $b_1, \dots, b_n$}) if it is
\begin{align*}
    \left\langle x_1, \ldots, x_{m}~
    \begin{array}{|c}
        b_i \cdot x_{1} = b_i \cdot x_{2} \quad (i=1,\ldots,n)
    \end{array}
    \right\rangle.
\end{align*}
We say that an $(m,n)$-presentation satisfies the \textit{$\partial$-condition} if there exist $n$ signs $\varepsilon_1, \dots, \varepsilon_n \in \{\pm 1\}$ such that
    \[
        \prod_{i=1}^n b_i^{-1} \sigma_1^{\varepsilon_i} b_i = 1_{m},
    \]
where $\sigma_i$ is Artin's generator of $B_m$ and $1_{m}$ is the unit element of $B_m$.

\begin{theorem}[\cite{Kamada1994-01}]\label{Theorem: Kamadas characterization}
    A group $G$ is the knot group of a $c$-component oriententable surface-link with the Euler characteristic $\chi$ if and only if $G$ satisfies the following conditions for some $m, n \geq 0$:
    \begin{enumerate}
        \item $G$ has an $(m,n)$-presentation satisfying the $\partial$-condition.
        \item It holds that $\chi = 2m - n$.
        \item $G/[G,G]$ is isomorphic to $\Z^{c}$.
    \end{enumerate}
    In particular, a group $G$ is the knot group of an orientable surface-link if and only if $G$ satisfies (1) for some $m, n \geq 0$.
\end{theorem}

In this paper, we give an analogue result for arbitrary surface-links including non-orientable ones.
A group presentation is called a $(2m,n)$-presentation \textit{with inverses} (associated with $b_1, \dots, b_n$) if it is
\[
    \left\langle x_1, \ldots, x_{2m}~
    \begin{array}{|cc}
        b_i \cdot x_{1} = b_i \cdot x_{2}, ~ x_{2j-1}=x_{2j}^{-1} \quad (i = 1,\ldots,n, ~ j = 1,\ldots,m)
    \end{array}
    \right\rangle,
\]
which is obtained from a $(2m, n)$-presentation by adding new $m$ relations $x_{2j-1} = x_{2j}^{-1}$ ($j = 1, \dots, m$).
We say that an $(2m,n)$-presentation satisfies the \textit{weak $\partial$-condition} if there exist $n$ signs $\varepsilon_1, \dots, \varepsilon_n \in \{1, -1\}$ such that
\[
    \prod_{i=1}^n b_i^{-1} \sigma_1^{\varepsilon_i} b_i \in K_{2m},
\]
where $K_{2m}$ is the Hilden subgroup of $B_{2m}$ defined in Section~\ref{Subsection: The plat closure of surface-links}.
A surface-link $F$ is called \textit{$(c, d)$-component} if $F$ consists of $c$ orientable surface-knots and $d$ non-orientable ones.
% The (non-orientable) genus of $(c,d)$-component surface-link $F$ means the integer $g$ satisfying $\chi(F) = 2(c+d) - g$, where $\chi(F)$ is the Euler characteristic of $F$.

\begin{theorem}\label{Main theorem}
    A group $G$ is the knot group of a $(c, d)$-component surface-link with the Euler characteristic $\chi$ if and only if $G$ satisfies the following conditions for some $m, n \geq 0$:
    \begin{enumerate}
        \item $G$ has an $(2m,n)$-presentation with inverses satisfying the weak $\partial$-condition.
        \item It holds that $\chi = 2m - n$.
        \item $G/[G,G]$ is isomorphic to $\Z^{c} \oplus (\Z/2)^{d}$.
    \end{enumerate}
    In particular, a group $G$ is the knot group of a surface-link if and only if $G$ satisfies (1) for some $m, n \geq 0$.
\end{theorem}

We will show this theorem using the fact that every surface-link is equivalent to the plat closure of some braided surface, which is discussed in Section~\ref{Subsection: The plat closure of surface-links}.

In Section~\ref{Section: Quandles and symmetric quandles}, we extend Theorem~\ref{Main theorem} to the knot symmetric quandles of surface-links.
A \textit{quandle} \cite{Joyce82,Mateev84} is a set $Q$ with a binary operation $*: Q \times Q \to Q$ satisfying the following conditions:
\begin{enumerate}
    \item For any $a \in Q$, we have $a*a = a$.
    \item For any $a, b \in Q$, there exists a unique element $c \in Q$ such that $c*b = a$.
    \item For any $a, b, c \in Q$, we have $(a*b)*c = (a*c) * (b*c)$.
\end{enumerate}

Quandles are usefull tools in the study of oriented links and oriented surface-links.
Specifically, the knot quandle is defined for oriented links and oriented surface-links (Section~\ref{Subsection: Quandles}).
It has been shown that the knot quandle is a stronger invariant than the knot group for distinguishing surface-links (\cite{Tanaka-Taniguchi2023-arXiv}).

To apply quandle theory to surface-links including non-orientable ones, Kamada \cite{Kamada2006} introduced symmetric quandles, which is a pair of a quandle $Q$ and a good involution $\rho$ of $Q$.
The knot symmetric quandle is defined for any surface-link, with details provided in Section~\ref{Subsection: Symmetric quandles}.
We present an algebraic characterization of the knot symmetric quandles of surface-links:

\begin{theorem}\label{Main Theorem2}
    A symmetric quandle $(Q, \rho)$ is the knot symmetric quandle of a $(c,d)$-component surface-link with the Euler characteristic $\chi$ if and only if $(Q, \rho)$ satisfies the following conditions for some $m,n \geq 0$:
    \begin{enumerate}
        \item $(Q, \rho)$ has an $(2m,n)$-presentation with inverses satisfying the weak $\partial$-condition.
        \item It holds that $\chi = 2m - n$.
        \item $Q$ consists of $2c + d$ connected components $X_1, \dots, X_c$, $Y_1,  \dots, Y_c$, and $Z_1, \dots, Z_d$ such that $\rho(X_i) = Y_i$ and $\rho(Z_j) = Z_j$ for each $i \in \{1, \dots, c\}$ and $j \in \{1, \dots, d\}$.
    \end{enumerate}
    In particular, a symmetric quandle $(Q, \rho)$ is the knot symmetric quandle of a surface-link if and only if $(Q, \rho)$ satisfies (1) for some $m, n \geq 0$.
\end{theorem}

The dihedral quandle $R_n$ is one of the most basic examples of quandles.
In Section~\ref{Section: knot symmetric quandles and dihedral quandles}, we provide presentations of $R_n$ as symmetric quandles.
Applying Theorem~\ref{Main Theorem2}, we show the following result:

\begin{theorem}\label{Theorem: dihedral quandle and knot symmetric quandle}
    For any $n \geq 1$ and any good involution $\rho$ of $R_n$, there exists a surface-link whose knot symmetric quandle is isomorphic to $(R_n, \rho)$.
\end{theorem}

% A \textit{$P^2$-knot} is a surface-knot homeomorphic to the projective plane, and a \textit{$P^2$-link} is a disjoint union of $P^2$-knots.
In Section~\ref{Section: Final remarks}, we discuss the relationship between symmetric quandles and $P^2$-irreducibility of surface-links.
We construct an infinite family of 2-component $P^2$-irreducible $P^2$-links (Theorem~\ref{Theorem: P2-irreducible P2-links}).

\section{Braids and Braided surfaces}\label{Section: Braided surfaces and plat closure}
\subsection{Braids and Braided surfaces}\label{Subsection: Braids and braided surfaces}
Let $m \geq 1$ be an integer, $D^2$ a 2-disk, $I = [0,1]$, and $p_2: D^2\times I \to I$ the projection of $D^2\times I$ onto the second factor.
We fix an $m$-point subset $X_m = \{x_1, \dots, x_m\}$ of $\mathrm{Int}(D^2)$ such that $x_1, \dots, x_m$ lie on a line of $D^2$ in this order.
An \textit{$m$-braid} is a union $\beta$ of $m$ curves in $D^2\times I$ such that
\begin{enumerate}
    \item the restriction map $p_2|_\beta: \beta \to I$ is a covering map of degree $m$ and
    \item $\partial \beta = X_m \times \{0,1\}$.
\end{enumerate}
Two $m$-braids are said to be \textit{equivalent} if they are isotopic in $D^2\times I$ rel $\partial$.
The \textit{braid group} $B_m$ of degree $m$ is the fundamental group of the configuration space $\mathcal{C}_m$ of $m$-point subsets of $\mathrm{Int}(D^2)$ with based point $X_m$, which has the following presentation:
\begin{align*}
    \left\langle \sigma_1, \ldots, \sigma_{m-1}~
    \begin{array}{|cc}
        \sigma_i \sigma_j \sigma_i = \sigma_j \sigma_i \sigma_j~ (|i-j|=1), \quad \sigma_i \sigma_k = \sigma_k \sigma_i ~ (|i-k| > 1)
    \end{array}
    \right\rangle,
\end{align*}
where $\sigma_i$ ($i = 1, \dots, m-1$) is the generator of $B_m$ due to Artin \cite{Artin1925}.
Each elements of $B_m$ are naturally identified with equivalence classes of $m$-braids.

Let $B^2$ be a 2-disk with a based point $y_0 \in \partial B^2$.
Let $\mathrm{pr}_1: D^2\times B^2 \to D^2$ and $\mathrm{pr}_2: D^2\times B^2 \to B^2$ denote projections of $D^2\times B^2$.
A \textit{braided surface} \cite{Rudolph1983} of degree $m$ is a compact surface $S$ embedded in $D^2\times B^2$ such that
\begin{enumerate}
    \item the restriction map $\pi_S := \mathrm{pr}_2|_S: S \to B^2$ is a branched covering map of degree $m$ and
    \item $S \cap D^2 \times \{y_0\} = X_m \times \{y_0\}$.
\end{enumerate}
A \textit{2-dimensional braid} \cite{Viro90} of degree $m$ is a braided surface $S$ of degree $m$ with $\partial S = X_m \times \partial B^2$.
% We denote by $\pi_S = \mathrm{pr}|_S$.
In this paper, every braided surface is assumed to be \textit{simple}, meaning that for every $y \in B^2$, the preimage $\pi_S^{-1}(y)$ consists of $m$ or $m-1$ points.
Two braided surfaces $S$ and $S'$ are said to be \textit{equivalent} if there exists
\begin{itemize}
    \item an isotopy $\{\Phi_t\}_{t \in [0,1]}$ of $D^2\times B^2$ carrying $S$ to $S'$, and
    \item an isotopy $\{h_t\}_{t \in [0,1]}$ of $B^2$
\end{itemize}
such that the following conditions hold for $t \in [0,1]$:
\begin{enumerate}
    \item $\mathrm{pr}_2 \circ \Phi_t = h_t \circ \mathrm{pr}_2$, and
    \item $\Phi_t|_{D^2 \times \{y_0\}} = \id$.
\end{enumerate}
Kamada \cite{Kamada1996} showed that $S$ and $S'$ are equivalent if and only if there exists an isotopy $\{\Phi_t\}_{t \in [0,1]}$ of $D^2\times B^2$ carrying $S$ to $S'$ such that $\Phi_t(S)$ is a simple braided surface for $t \in [0,1]$.

\subsection{Braid systems of braided surfaces}\label{Subsection: Braid systems of braided surfaces} % **
Let $S$ be a braided surface of degree $m$.
The branch locus of $\pi_S: S \to B^2$ is denoted by $\Sigma(S) = \{x_1, \dots, x_n\}~(\subset B^2)$, where $n \geq 0$.

The \textit{braid monodromy} of $S$ is the homomorphism $\rho_S: \pi_1(B^2\setminus \Sigma(S), y_0) \to \pi_1(\mathcal{C}_m, X_m)$ defined as follws:
For each $y \in B^2\setminus \Sigma(S)$, the image $\mathrm{pr}_1(\pi_S^{-1}(y))$ is a point of $\mathcal{C}_m$.
For a loop $\gamma: (I, \partial I) \to (B^2\setminus \Sigma(S), y_0)$, the loop $\rho_S(\gamma): (I, \partial I) \to (\mathcal{C}_m, X_m)$ is defined by $\rho_S(\gamma)(t) := \mathrm{pr}_1(\pi_S^{-1}(\gamma(t)))$.
Then, the braid monodromy is defined as the homomorphism $\rho_S: \pi_1(B^2\setminus \Sigma(S), y_0) \to B_m = \pi_1(\mathcal{C}_m, X_m)$ sending $[\gamma]$ to $[\rho_S(\gamma)]$.

It is known that $\pi_1(B^2\setminus \Sigma(S), y_0)$ is isomorphic to the free group on $n$ letters.
To specify a generating system of $\pi_1(B^2\setminus \Sigma(S), y_0)$, we use a \textit{Hurwitz arc system} in $B^2$ (with the base point $y_0$) which is an $n$-tuple $\mathcal{A} = (\alpha_1, \cdots, \alpha_r)$ of oriented simple arcs in $B^2$ such that
\begin{enumerate}
   \item for each $i =1, \ldots, r$, $\alpha_i \cap \partial D_2 = \partial \alpha_i \cap \partial D_2 = \{y_0\}$ which is the terminal point of $\alpha_i$,
   \item $\alpha_i \cap \alpha_j = \{y_0\}$ ($i\neq j$), and
   \item $\alpha_1, \ldots,  \alpha_r$ appear in this order around the base point $y_0$.
\end{enumerate}
The \textit{starting point set} of $\mathcal{A}$ is the set of initial points of $\alpha_1, \ldots, \alpha_r$.

Let $\mathcal{A} = (\alpha_1, \cdots, \alpha_n)$ be a Hurwitz arc system in $B^2$ with the starting point set $\Sigma(S)$.
Let $N_i$ be a (small) regular neighborhood of the starting point of $\alpha_i$ and $\overline{\alpha_i}$ an oriented arc obtained from $\alpha_i$ by restricting to $D_2\setminus \Int N_i$.
For each $\alpha_i$, let $\gamma_i$ be a loop $\overline{\alpha_i}^{-1} \cdot \partial N_i \cdot \overline{\alpha_i}$ in $D_2\setminus \Sigma(S)$ with base point $y_0$, where $\partial N_i$ is oriented counter-clockwise.
The $n$-tuple $(\gamma_1, \dots, \gamma_n)$ is called a \textit{Hurwitz loop system} in $B^2$ with the base point $y_0$, and $\pi_1(B^2\setminus \Sigma(S), y_0)$ is the free group generated by $[\gamma_1], [\gamma_2], \ldots, [\gamma_n]$.
\begin{definition}
    Let $\gamma_i$ as above.
    A \textit{braid system} of $S$ is an $n$-tuple $(\beta_1, \dots, \beta_n)$ with $\beta_i = \rho_S(\gamma_i)$.
\end{definition}
The choice of a braid system depends on the Hurwitz arc system. Rudolph \cite{Rudolph1983} showed that an $n$-tuple $(\beta_1, \dots, \beta_n)$ of $m$-braids is  a braid system of some $S$ if and only if $\beta_i$ is a conjugate of $\sigma_1$ or $\sigma_1^{-1}$.
We notice that for $i \in \{ 1, \dots, n-1 \}$ and $\varepsilon \in \{ \pm 1\}$, a conjugate of $\sigma_i^{\varepsilon}$ is a conjugate of $\sigma_1^{\varepsilon}$.
Such an $n$-tuple is simply called a braid system.
For a braid system $b = (\beta_1, \dots, \beta_n)$ and $i \in \{1,\dots, n-1\}$, we define new braid systems
\[
    b_1 = (\beta_1, \dots, \beta_{i-1},~
    \beta_i \, \beta_{i+1} \,\beta_i^{-1},~
    \beta_i,~
    \beta_{i+2}, \dots, \beta_n),
    \quad
    b_2 = (\beta_1, \dots, \beta_{i-1},~
    \beta_{i+1},~
    \beta_{i+1}^{-1} \, \beta_{i} \,\beta_{i+1},~
    \beta_{i+2}, \dots, \beta_n).
\]
Then, we say that $b_1$ and $b_2$ are obtained from $b$ by a \textit{slide action}.
% This operation is called a \textit{slide action}.
Two braid systems are said to be \textit{slide equivalene} (or \textit{Hurwitz equivalene}) if one can be obtained from the other by a finite sequance of slide actions.

\begin{proposition}[\cite{Rudolph1983}]
    There exists an one to one natural correspondance between the set of equivalene classes of (simple) braided surfaces degree $m$ and the set of slide equivalene classes of braid systems of degree $m$:
    \[
        \left\{\mbox{Braided surfaces}\right\}/ \mbox{equivalene} \quad \stackrel{1:1}{\longleftrightarrow} \quad \left\{ \mbox{braid systems} \right\}/\mbox{slide equivalene}.
    \]
\end{proposition}

The braid group $B_m$ acts on the free group $F_m$ as follows:
The braid group $B_m$ is identified with the mapping class group $\mathrm{MCG}_\partial(D^2, X_m)$, the group of isotopy classes of homeomorphisms $\phi: D^2\setminus X_m \to D^2\setminus X_m$ that fix $\partial D^2$ pointwise.
An element $b \in B_m$ corresponds to an isotopy class $[\phi]$, and the action of $[\phi]$ on an element $[\gamma] \in \pi_1(D^2\setminus X_m)$ is defined as $[\phi \circ \gamma]$.

A Hurwitz arc system in $D^2$ with the starting point set $X_m$ induces generators $x_1, \dots, x_m$ for $\pi_1(D^2\setminus X_m)$, making it the free group $F_m$ on $x_1, \dots, x_m$.
The action of $B_m$ on $F_m$ is defined similaly via homeomorphisms $\phi$ as desctibed above.
Explicitly, the images of $\sigma_i \cdot x_j$ and $\sigma_i^{-1} \cdot x_j$ are given by:
\begin{align*}
    \sigma_i \cdot x_j ~=~ \begin{cases}
        x_i \, x_{i+1} \, x_i^{-1} &(j = i),\\
        x_i     &(j = i+1),\\
        x_j     &(\mbox{otherwise}),
    \end{cases} \qquad
    \sigma_i^{-1} \cdot x_j ~=~ \begin{cases}
        x_{i+1}     &(j = i),\\
        x_{i+1}^{-1} \, x_{i} \, x_{i+1} &(j = i+1),\\
        x_j     &(\mbox{otherwise}).
    \end{cases}
\end{align*}
For braids $b_1, b_2 \in B_m$, the action of $(b_1 b_2)$ on $w \in F_m$ is given by $(b_1 b_2)\cdot w = b_2 \cdot (b_1 \cdot w)$.

\subsection{The plat closure of braided surfaces}\label{Subsection: The plat closure of surface-links}

% We recall the definition of the plat closure \cite{Yasuda21} of adequate braided surfaces.

A \textit{wicket} \cite{Brendle-Hatcher2008} is a semi-circle in $D^2 \times I$ that meets $D^2 \times \{0\}$ orthogonally at its endpoints in $\mathrm{Int}D^2$.
A \textit{configuration of $m$ wicket} is a disjoint union of $m$ wickets.
Let $w_0$ be the configuration of $m$ wickets such that each wicket has the boundary $\{x_{2i-1}, x_{2i}\} \subset X_{2m}$ ($i \in \{1, \dots, m\}$), and let $\mathcal{W}_m$ denote the space consisting of configurations of $m$-wickets.
For a loop $f: (I, \partial I) \to (\mathcal{W}_m, w_0)$, the $2m$-braid $\beta_f$ is defined as
\[
    \beta_f ~=~ \bigcup_{t \in I} \partial f(t) \times \{t\} ~\subset~ (D^2 \times \{0\}) \times I = D^2\times I.
\]
A $2m$-braid $\beta$ is \textit{adequate} if $\beta = \beta_f$ for some loop $f: (I, \partial I) \to (\mathcal{W}_m, w_0)$.

The \textit{Hilden subgroup} $K_{2m}$ of $B_{2m}$ is the subgroup generated by $\sigma_1$, $\sigma_2 \sigma_1 \sigma_3 \sigma_2$ and $\sigma_{2k} \sigma_{2k-1} \sigma^{-1}_{2k+1} \sigma^{-1}_{2k}$ for $k \in \{1, \dots, m-1\}$.
We remark that $\sigma_{2j-1}$ and $\sigma_{2k} \sigma_{2k-1} \sigma_{2k+1} \sigma_{2k}$ also belong to $K_{2m}$ for $j \in \{1, \dots, m \}$ and $k \in \{1, \dots, m-1\}$.
Brendle-Hatcher \cite{Brendle-Hatcher2008} proved that $K_{2m}$ is precisely the subgroup generated by adequate $2m$-braids.

We assume that $B^2$ is the unit disk on $\C$ with $y_0 = 1 \in \partial B^2$ and $D^2 \subset \R^2$.
We identify $\C$ with $\R^2$ so that $D^2 \times B^2 \subset \R^4$.
For a braided surface $S$ of degree $m$, the $m$-braid $\beta_S$ are defined as follows:
\[
    \beta_S = \bigcup_{t \in I} \mathrm{pr_1}(\pi_S^{-1}(e^{2\pi \sqrt{-1}t})) \times \{t\} \subset D^2 \times I.
\]
We call $\beta_S$ the \textit{boundary braid} of $S$.
A braided surface is \textit{adequate} if the boundary braid $\beta_S$ is adequate.
We remark that the degree of an adequate braided surface is even.

Let $S$ be an adequate braided surface of degree $2m$, and we define the plat closure of $S$ as follows:
For $t \in I = [0,1]$, we define an interval $J_t = \{ re^{2\pi \sqrt{-1}t} \in \C ~|~ 1 \leq r \leq 2\}$, where $J_0 = J_1$.
The union of $J_t$ ($t \in I$) is an annulus with the boundary containing $\partial B^2$.
Since $\beta_S$ is adequate, there exists a unique loop $f: (I, \partial I) \to (\mathcal{W}_m, w_0)$ such that $\beta_S = \beta_f$.
For each $t \in I$, we define $w_t$ as the configuration of $m$ wickets $f(t)$ in $D^2 \times J_t$ such that $\partial w_t \subset D^2 \times \{e^{2\pi \sqrt{-1}t}\}$.
Then, the union of $w_t$ ($t \in I$), denoted by $A_S$, is the compact surface embedded in $(D^2\times B^2)^c = \R^4 \setminus \Int(D^2\times B^2)$ such that $\partial A_S = \partial S = S \cap A_S$.

\begin{definition}
    The \textit{plat closure} of $S$ is the union of $S$ and $A_S$, denoted by $\widetilde{S}$.
\end{definition}

\begin{proposition}[\cite{Yasuda21}]\label{theorem: plat form presentation}
    Every surface-link is ambiently isotopic to the plat closure of some braided surface.
\end{proposition}

\begin{proposition}[\cite{Yasuda24}]\label{Proposition: knot group of the plat form}
    Let $S$ be an adequate braided surface of degree $2m$, and let $(\beta_1, \dots, \beta_n)$ a braid system of $S$ such that $\beta_i = b_i^{-1} \sigma_1^{\,\varepsilon_i} b_i$ ($\varepsilon_i \in \{ 1, -1\}$).
    Then, the knot group of the plat closure $\widetilde{S}$ has a $(2m,n)$-presentation with inverses associated with $b_1, \dots, b_n$:
   \[
        \left\langle
        \begin{array}{c|c}
            x_1, \ldots, x_{2m} &
            b_i\cdot x_{1} = b_i\cdot x_{2}, \quad x_{2j-1}=x_{2j}^{-1} \quad (i \in \{1,\ldots,n\}, j \in \{1,\ldots,m\})
        %    \mathrm{Artin}(b_i)(x_{1}) = \mathrm{Artin}(b_i)(x_{2}), ~ x_{2j-1}=x_{2j}^{-1} ~(i \in \{1,\ldots,n\}, j \in \{1,\ldots,m\})
        \end{array}
      \right\rangle.
     \]
    % where $b_1, \dots, b_n$ are $m$-braids such that $(b_1, \dots, b_n)$ represents $S$ as a braid system.
\end{proposition}

\begin{proof}
    Let $S$ be an adequate braided surface of degree $2m$ and $\beta_S$ the boundary braid of $S$.
    From \cite{Kamada1994-01}, the knot group $G(S) = \pi_1(D^2\times B^2 \setminus S)$ has $(2m,n)$-presentation
    \[
        G(S) ~=~
        \left\langle
        \begin{array}{c|c}
            x_1, \ldots, x_{2m} &
            b_i \cdot x_{1} = b_i \cdot x_{2}~(i \in \{1,\ldots,n\})
        \end{array}
        \right\rangle.
    \]
    Let $L$ denote the intersection of $S$ and $A_S$, which is a link in the 3-sphere $\partial (D^2\times B^2)$.
    Since $L$ is the closure of $\beta_S$, the knot group $G(L)$ has a presentation
    \[
        G(L) ~=~
        \left\langle
            \begin{array}{c|c}
                x_1, \ldots, x_{2m} &
                x_k = \beta_S \cdot x_k ~ (k = 1, \dots, 2m)
            \end{array}
        \right\rangle.
    \]
    By applying \cite[Lemma 4.7]{Yasuda24} for the case of groups, the knot group of $G(A_S) = \pi_1((D^2\times B^2)^c \setminus A_S)$ has a presentation
    \[
        G(A_S) ~=~
        \left\langle
        \begin{array}{c|c}
            x_1, \ldots, x_{2m} &
            x_k = \beta_S \cdot x_k, \quad
            x_{2j-1} = x_{2j}^{-1} ~ (k = 1, \dots, 2m, j = 1, \dots, m)
        \end{array}
        \right\rangle.
    \]

    Finally, the homomorphism $\iota_S: G(L) \to G(S)$ (or $\iota_{A_S}: G(L) \to G(A_S)$) induced from the natural inclusion map sends $x_k \in G(L)$ to $x_k \in G(S)$ (or $x_k \in G(A_S)$), respectively.
    Hence we obtain the presentation in Proposition~\ref{Proposition: knot group of the plat form} by van Kampen's theorem.
\end{proof}

\subsection{Proof of Theorem~\ref{Main theorem}} % *
We divide the proof into two parts: the ``only if part'' and the ``if part''.

\noindent\textbf{(Proof of the only if part)}:
Let $G = G(F)$ be the knot group of a $(c, d)$-component genus $g$ surface-link $F$.
From Theorem~\ref{theorem: plat form presentation}, $F$ is ambiently isotopic to the plat closure of an adequate braided surface $S$ of degree $2m$.
Let $(\beta_1, \dots, \beta_n)$ be a braid system of $S$ such that $\beta_i = b_i^{-1} \sigma_1^{\,\varepsilon_i} b_i$ ($\varepsilon_i \in \{ 1, -1\}$).
By Proposition~\ref{Proposition: knot group of the plat form}, $G$ admits a $(2m, n)$-presentation with inverses associated with $b_1, \dots, b_n$.
Since $\beta_S$ is adequate, the $(2m, n)$-presentation with inverses of $G$ satisfies the weak $\partial$-condition:
\begin{align*}
    \prod_{i=1}^n b_i^{-1} \sigma_1^{\,\varepsilon_i} b_i ~&=~ \prod_{i=1}^n \beta_i ~=~ \prod_{i=1}^n \rho_S\left(\left[\gamma_i\right]\right) ~=~ \rho_S([\prod_{i=1}^n \gamma_i]) ~=~ \rho_S([\partial B^2])\\
        &=~ \beta_S ~ \in ~ K_{2m},
\end{align*}
where $(\gamma_1, \dots, \gamma_n)$ is a Hurewitz loop system in $B^2$.
It holds that $\chi(\widetilde{S}) = 2m - n$ so that $G$ satisfies the condition (2).
Using the Alexander Duality Theorem, we know that $H_1(\R^4 \setminus F) \cong \Z^{c} \oplus (\Z/2)^{d}$.
Since $H_1(\R^4\setminus F)$ is the abelianization of $G(F)$, $G$ satisfies the condition (3).

% Next, we prove the if part.
\noindent\textbf{(Proof of the if part)}:
Suppse $G$ is a group satisfying the conditions (1) -- (3) of Theorem~\ref{Main theorem} for some $m,n \geq 0$.
Specifically, $G$ admits a $(2m, n)$-presentation with inverses associated with $2m$-braids $b_1, \dots, b_n$, and there exists signs $\varepsilon_1, \dots, \varepsilon_n$ such that
\begin{align*}
    \prod_{i=1}^n b_i^{-1} \sigma_1^{\,\varepsilon_i} b_i ~\in~ K_{2m}.
\end{align*}
Let $S$ be a braided surface of degree $2m$ associated with a braid system $(\beta_1, \dots, \beta_n)$ with $\beta_i = b_i^{-1} \sigma_1^{\,\varepsilon_i} b_i$.
Since $\beta_S = \prod_{i=1}^n b_i^{-1} \sigma_1^{\,\varepsilon_i} b_i \in K_{2m}$, $S$ can be chosen as an adequate braided surface.
We define the surface-link $F$ as the plat closure of $S$.
By Proposition~\ref{Proposition: knot group of the plat form}, the knot group of $F$ is isomorphic to $G$.
Furthermore, $F$ is $(c,d)$-component and genus $g$ because
\begin{align*}
    H_1(\R^4\setminus F) \cong G/[G,G] \cong \Z^{c} \oplus (\Z/2)^{d} \quad \mbox{and} \quad \chi(\widetilde{S}) = 2m-n = 2(c+d)-g.
\end{align*}
This completes the proof of Theorem~\ref{Main theorem}.
\qed

The author \cite{Yasuda21} also showed that every orientable surface-link is ambiently isotopic to the plat closure of a 2-dimensional braid.
Consequently, we have a similar characterization of knot groups of orientable surface-links:

\begin{theorem}\label{Theorem: Characterization of knot grp in genuine plat form}
    A group $G$ is the knot group of a $c$-component orientable surface-link with the Euler characteristic $\chi$ if and only if $G$ satisfies the following conditions for some $m, n \geq 0$:
    \begin{enumerate}
        \item $G$ has an $(2m,n)$-presentation with inverses satisfying the $\partial$-condition.
        \item It holds that $\chi = 2m - n$.
        \item $G/[G,G]$ is isomorphic to $\Z^{c}$.
    \end{enumerate}
\end{theorem}

% It is unknown that every surface-link with the normal Euler number zero is ambiently isotopic to the plat closure of some 2-dimensional braid.
% If this conjecture is true, then the above Theorem can be

\section{Quandles and symmetric quandles}\label{Section: Quandles and symmetric quandles}
\subsection{Quandles}\label{Subsection: Quandles}
% In this section, we characterize knot symmetric quandles of surface-links.
A \textit{quandle} \cite{Joyce82,Mateev84} is a set $Q$ with a binary operation $*: Q\times Q \to Q$ on $Q$ satisfying the following three axioms:
\begin{enumerate}
    \item[(Q1)] For any $a \in Q$, we have $a*a = a$.
    % \item[(Q2)] For any $b \in Q$, the map $S_b: Q \to Q; a \mapsto a*b$ is bijective.
    \item [(Q2)] For any $a, b \in Q$, there exists a unique element $c \in Q$ such that $c*b = a$.
    \item[(Q3)] For any $a, b, c \in Q$, we have $(a*b)*c = (a*c) * (b*c)$.
\end{enumerate}
A \textit{rack} \cite{Fenn-Rourke1992} is a set $Q$ with a binary operation on $Q$ satisfying (Q2) and (Q3).

A quandle homomorphism is a map $f: Q \to Q'$ such that $f(a*b) = f(a)*f(b)$ for any $a, b \in Q$.
For $b \in Q$, the map $S_b: Q \to Q$ sending $a \in Q$ to $a*b \in Q$ is a quandle homomorphism.
Moreover, $S_b$ is a quandle automorphism.
The \textit{inner automotphism group} of $Q$, denoted by $\mathrm{Inn}(Q)$, is the subgroup of $\mathrm{Aut}(Q)$ generated by $S_b$ for all $b \in Q$.
A quandle $Q$ is called  \textit{connected} if $\mathrm{Inn}(Q)$ acts transitively on $Q$.
Similaly, an orbit of $Q$ under the action of $\mathrm{Inn}(Q)$ is called a \textit{connected component} of $Q$.

For a quandle $(Q, *)$, we define the binary operation $\overline{*}$ on $Q$ by $a\overline{*}b = c$, where $c \in Q$ is a unique element obtained from the axiom (Q2).
Then $(Q, \overline{*})$ is again a quandle.
We call $\overline{*}$ the \textit{dual operation} of $*$.

When expressing elements of $Q$, we sometimes use the Fenn-Rourke notation \cite{Fenn-Rourke1992}:
Let $\mathrm{F}(Q)$ denote the free group generated by elements of $Q$.
For $a, b \in Q$, the symbol $a^b$ denotes $a*b$, and for a word $w = b_1^{\varepsilon_1} b_2^{\varepsilon_2} \cdots b_n^{\varepsilon_n} \in \mathrm{F}(Q)$, the symbols $a^w$ denotes the element $(\cdots((a*^{\varepsilon_1} b_1)*^{\varepsilon_2} b_2 ) * \cdots)*^{\varepsilon_n} b_n \in Q$, where $*^{1} = *$ and $*^{-1} = \overline{*}$.

\begin{example}[Free quandle]
    Let $A$ be a non-empty set.
    Then $\mathrm{FR}(A) = A \times \mathrm{F}(A)$ is a rack by a binary operation $(x, w) * (y, u) = (x, w u^{-1} y u)$ for $(x, w), (y, u) \in \mathrm{FR}(A)$.
    We call $\mathrm{FR}(A)$ the \textit{free rack} on $A$.

    Let $\sim_q$ be the equivalence relation on $\mathrm{FR}(A)$ generated by $(x, w) \sim_q (x, xw)$ for $x \in A$ and $w \in \mathrm{F}(A)$.
    Then the rack operation on $\mathrm{FR}(A)$ induces a quandle operation on $\mathrm{FQ}(A) := \mathrm{FR}(A)/\sim_q$.
    We call $\mathrm{FQ}(A)$ the \textit{free quandle} on $A$.
    We write $x^w$ for $(x, w) \in \mathrm{FQ}(A)$ and $x$ for $(x,1)$, where $1 \in \mathrm{F}(A)$ is the unit element.

    For a subset $R$ of $\mathrm{FQ}(A) \times \mathrm{FQ}(A)$, the quandle $\left\langle A ~|~ R \right\rangle_{q}$ is defined in a manner similar to the case of group presentations.
    The quandle $\left\langle A ~|~ R \right\rangle_{q}$ is called a \textit{quandle presentation} of a quandle $Q$ if $Q$ is isomorphic to $\left\langle A ~|~ R \right\rangle_{q}$.
    (See \cite{Fenn-Rourke1992} for details.)
\end{example}

\begin{example}[Knot quandle]
    Let $K$ be a properly embedded oriented $n$-submanifold in an oriented $(n+2)$-manifold $M$, $N(K)$ a regular neighborhood of $K$ in $M$, and $E(K) = \mathrm{cl}(M \setminus N(K))$.
    We take a based point $p$ of $E(K)$.
    A \textit{noose} of $K$ is a pair $(D,\alpha)$ of a meridional disk $D$ of $K$ and an oriented arc $\alpha$ in $E(K)$ connecting from a point of $\partial D$ to $p$.
    Since $K$ and $M$ are oriented, the meridional disk $D$ can be oriented such that the tuple of orientations of $K$ and $D$ represents that of $M$.
    The induced orientation of $\partial D$ from that of $D$ is called the \textit{positive direction} of $\partial D$.
    We denote by $Q(M, K, p)$ the set of homotopy classes $[(D, \alpha)]$ of all noose of $K$.
    The \textit{knot quandle} (or \textit{fundamental quandle}) of $K$ is $Q(M, K, p)$ with the binary operation $*$ defined by
    \[
        [(D, \alpha)] * [(D', \alpha')] = [(D, \alpha\cdot \alpha'^{-1} \partial D' \alpha')],
    \]
    where $\partial D'$ is a meridional loop starting from the initial point of $\alpha'$ and going along $\partial D'$ in the positive direction.
    % based on the initial point of $\alpha'$ induced from the natural orientation of $D$.
    See Figure~\ref{Figure: operation of knot quandle} for $n = 1$.
    We notice that the knot quandle is independent of the choice of the based point $y$ when $M$ is connected,
    Then we denote it by $Q(M, K)$ or $Q(K)$ simply.
    It is known that the numbers of connected components of $K$ and $Q(K)$ are equal.

    \begin{figure}[h]
        \centering
        \includegraphics[width = 0.6\hsize]{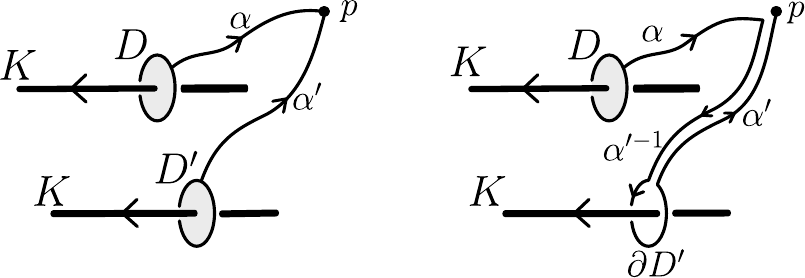}
        \caption{Nooses of $K$ (left) and an operation result of the knot quandle (right)}
        \label{Figure: operation of knot quandle}
    \end{figure}

    We identify the braid group $B_m$ with the mapping class group $\mathrm{MCG}_\partial(D^2, X_m)$.
    For $b = [\phi] \in \mathrm{MCG}_\partial(D^2, X_m)$ and $x = [(D, \alpha)] \in Q(D^2, X_m)$, we define $b\cdot x = [(\phi(D), \phi(\alpha))] \in Q(D^2, X_m)$.
    Then, this gives an action on the knot quandle $Q(D^2, X_m)$ by the braid group $B_m$.

    The knot quandle $Q(D^2, X_m)$ is isomorphic to the free quandle on $m$ letters.
    Precisely, a Hurwitz arc system in $D^2$ with the starting point set $X_m$ induces $m$ nooses of $X_m$ such that each noose has a meridional disk of $x_i \in X_m$ ($i = 1, \dots, m$).
    We denote homotopy classes of these nooses by same symbols $x_1, \dots, x_m$.
    Then, the free quandle $\mathrm{FQ}(X_m)$ is isomorphic to $Q(D^2, X_m)$ by the isomorphism sending $x_i \in \mathrm{FQ}(X_m)$ to $x_i \in Q(D^2, X_m)$.
    Therefore, we obtain the action on $\mathrm{FQ}(X_m)$ by $B_m$ from $Q(D^2, X_m)$.
    Explicitly, we have
    \begin{align*}
        \sigma_i \cdot x_j ~=~ \begin{cases}
            x_{i+1} \,\overline{*}\, x_i &(j = i),\\
            x_i     &(j = i+1),\\
            x_j     &(\mbox{otherwise}).
        \end{cases} \qquad
        \sigma_i^{-1} \cdot x_j ~=~ \begin{cases}
            x_{i+1}     &(j = i),\\
            x_{i} * x_{i+1} &(j = i+1),\\
            x_j     &(\mbox{otherwise}).
        \end{cases}
    \end{align*}
\end{example}

\subsection{Symmetric quandles}\label{Subsection: Symmetric quandles}
A \textit{good involution} of a quandle $Q$ is an involution $\rho: X \to X$ (i.e., $\rho\circ\rho = \id_Q$) such that for any $a, b \in Q$, we have $\rho(a*b) = \rho(a) * b$ and $a*\rho(b) = a \overline{*} b$.
A \textit{symmetric quandle} is a pair $X = (Q, \rho)$ of a quandle $Q$ and its good involution $\rho$.
A symmetric quandle homomorphism from $(Q, \rho)$ to $(Q', \rho')$ is a quandle homomorphism $f: Q \to Q'$ such that $f \circ \rho = \rho' \circ f$.

\begin{example}[Dihedral quandles]\label{Example: dihedral quandle}
    The \textit{dihedral quandle} $R_n$ is $\Z/n\Z$ with $a*b = 2b-a$.
    Kamada and Oshiro \cite{Kamada-Oshiro2010} classified the good involutions of dihedral quandles:
    \begin{enumerate}
        \item
        When $n \equiv 1, 3 ~(\mathrm{mod}~4)$, the idetity map $\id_{R_n}$ is the only good involution of $R_n$.
        \item
        When $n \equiv 2 ~(\mathrm{mod}~4)$, the good involution is either the idetity map $\id_{R_n}$ or the \textit{antipodal map} $\rho_A$ which is defined by $\rho_A(a) = a + n/2$.
        \item
        When $n \equiv 0 ~(\mathrm{mod}~4)$, the good involution is either the idetity map $\id_{R_n}$, the antipodal map $\rho_A$, or one of the two types of \textit{half-antipodal maps}, denoted by $\rho_{HA}$ and $\rho'_{HA}$, which are defined by
        \begin{align*}
            \rho_{HA}(a) ~=~ \begin{cases}
                a + n/2 &(a\mbox{: even})\\
                a &(a\mbox{: odd})
            \end{cases}, \quad
            \rho'_{HA}(a) ~=~ \begin{cases}
                a &(a\mbox{: even})\\
                a + n/2 &(a\mbox{: odd})
            \end{cases}.
        \end{align*}
        We remark that $(R_{4n}, \rho_{HA})$ and $(R_{4n}, \rho'_{HA})$ are isomorphic as symmetric quandles.
        % by $\rho_{\mathrm{hAnt}}(a) = a$ for odd $a$ and $\rho_\mathrm{hAnt}(a) = a + k$ for even $a$.
    \end{enumerate}

    A \textit{kei} (or a \textit{involutory quandle}) is a quandle $Q$ such that $(a * b) *b = a$ for any $a, b \in Q$.
    The dihedral quandle is a kei.
    They also proved that $Q$ is a kei if and only if the identity map $\id_Q$ is a good involution of $Q$.
\end{example}

\begin{example}[Double of a quandle \cite{Kamada2006,Kamada-Oshiro2010}]\label{Example: the double of quandle}
    Let $Q$ be a quandle and $\overline{Q} = \{\overline{a} ~|~ a \in Q\}$ a copy of $Q$.
    We extend the binary operation $*$ on $Q$ into the disjoint union $D(Q) = Q \cup \overline{Q}$:
    \[
        a*\overline{b} := a\overline{*}b, \quad \overline{a}*b := \overline{a*b}, \quad \overline{a}*\overline{b} := \overline{a\overline{*}b}.
    \]
    Then $(D(Q), *)$ is a quandle.
    We define a good involution $\rho$ on $D(Q)$ by $\rho(a) = \overline{a}$.
    The symmetric quandle $(D(Q), \rho)$ is called the \textit{double} of $Q$.

    The \textit{free symmetric quandle} on a set $A$, denoted by $\mathrm{FSQ}(A)$, is the double of the free quandle $\mathrm{FQ}(A)$ on $A$.
    The underlying set of $\mathrm{FSQ}(A)$ is $(A \cup \overline{A}) \times \mathrm{F}(A)$ under the identification $\overline{a} \in \overline{A}$ with $a^{-1} \in \mathrm{F}(A)$.

    Oshiro showed \cite{Oshiro2020-preprint} that the double of $R_{2n+1}$ is the dihedral quandle $R_{4n+2}$ with the antipodal map.
\end{example}

\begin{example}[Knot symmetric quandles \cite{Kamada2006,Kamada-Oshiro2010}]
    Let $K$ be a properly embedded $n$-submanifold in an $(n+2)$-manifold $M$.
    We take a based point $p$ of $E(K)$.
    An \textit{oriented noose} of $K$ is a pair $(D,\alpha)$ of an oriented meridional disk $D$ of $K$ and an oriented arc $\alpha$ in $E(K)$ connecting from a point of $\partial D$ to $p$.

    $\widetilde{Q}(M, K, p)$ (or $\widetilde{Q}(K, p)$ simply) denotes the set of homotopy classes $[(D, \alpha)]$ of all oriented nooses of $K$.
    We remark that $[(D, \alpha)]$ and $[(-D, \alpha)]$ are different homotopy classes if $K$ is orientable, where $-D$ is $D$ with the reversed orientation.
    The \textit{full knot quandle} of $K$ is $\widetilde{Q}(M, K, p)$ with the binary operation $*$ defined by
    \[
        [(D, \alpha)] * [(D', \alpha')] = [(D, \alpha\cdot \alpha'^{-1} \partial D' \alpha')].
    \]
    We denote $\widetilde{Q}(M, K, p)$ by $\widetilde{Q}(K)$ simply when $M$ is connected.
    The good involution $\rho_K$ of $\widetilde{Q}(M, K)$ is defined by $\rho_K([(D, \alpha)]) = [(-D, \alpha)]$.
    The \textit{knot symmetric quandle} (or \textit{fundamental symmetric quandle}) of $K$, denoted by $X(M, K)$ (or $X(K)$ simply), is the pair of $\widetilde{Q}(K)$ and $\rho_K$.

    When $K$ is orientable, then $X(K)$ is the double of $Q(K)$.
    When $K$ is non-oriententable and connected, then $X(K)$ is also connected as a quandle.

    The knot symmetric quandle $X(D^2, X_m)$ is isomorphic to the free symmetric quandle $\mathrm{FSQ}(X_m)$.
    % We denote by $b\cdot x$ the natural action on $x \in \mathrm{FSQ}_m$ by $b \in B_m$.
    The action on the free quandle $\mathrm{FQ}(X_m)$ by the braid group $B_m$ extends the action on the free symmetric quandle $\mathrm{FSQ}(X_m)$ by $B_m$ defined by $b \cdot \overline{x_j} = \overline{b \cdot x_j}$.
\end{example}

\subsection{Symmetric quandle presentation}

In this section, we recall presentations of symmetric quandles.

Let $A$ be a non-empty set, and $R$ be a subset of $\mathrm{FSQ}(A) \times \mathrm{FSQ}(A)$.
Then, we define a subset $\langle \langle R \rangle \rangle_{\mathrm{sq}}$ of $\mathrm{FSQ}(A) \times \mathrm{FSQ}(A)$ by repeating the following moves:
\begin{enumerate}
    \item[(E1)] For every $x\in \mbox{FSQ}(A)$, add $(x,x)$ to $R$.
    \item[(E2)] For every $(x,y)\in R$, add $(y,x)$ to $R$.
    \item[(E3)] For every $(x,y), (y,z) \in R$, add $(x,z)$ to $R$.
    \item[(R1)] For every $(x,y)\in R$ and $a\in A$, add $(x^a,y^a), (x^{-a},y^{-a})$ to $R$.
    \item[(R2)] For every $(x,y)\in R$ and $t\in \mbox{FSQ}(A)$, add $(t^x,t^y), (t^{-x},t^{-y})$ to $R$.
    \item[(S)] For every $(x,y) \in R$, add $(x^{-1}, y^{-1})$ to $R$.
 \end{enumerate}

Then $\langle\langle R \rangle\rangle_{sq}$ is called the \textit{set of symmetric quandle consequences} of $R$, which is the smallest congruence containing $R$ with respect to the axioms of a symmetric quandle (\cite{Kamada2014}).
An element of $\langle\langle R \rangle\rangle_{sq}$ is called a \textit{consequence} of $R$.
Let $\langle A ~|~ R \rangle_{\mathrm{sq}}$ denote the quotient set $\mathrm{FSQ}(A)/ \langle\langle R \rangle\rangle_{sq}$.
The quandle operation and the good involution of $\mathrm{FSQ}(A)$ induce a quandle operation and a good involution on $\langle A ~|~ R \rangle_{\mathrm{sq}}$.

\begin{definition}
    The symmetric quandle $\langle A ~|~ R \rangle_{\mathrm{sq}}$ is a \textit{symmetric quandle presentation} of a symmetric quandle $X$ if $\langle A ~|~ R \rangle_{\mathrm{sq}}$ is isomorphic to $X$.
\end{definition}

In a symmetric quandle presentation, the symbol $\overline{x}$ denotes the image of $x$ by the good involution.
The notions and properties of symmetric quandle presentations are discussed in \cite{Kamada2014,Yasuda24}.
% Notions of presentations of symmetric quandles are discussed in \cite{Kamada2014}.

It is known that two finite presentations of a group are related by a finite sequence of Tietze transformations.
This result was extended for finite presentations of racks \cite{Fenn-Rourke1992} and of symmetric quandles \cite{Yasuda24}.

\begin{proposition}[\cite{Yasuda24}]\label{Proposition: }
    Any two finite presentations of a symmetric quandle are related by a finite sequence of Tietze transformations (T1) -- (T4):
    \begin{enumerate}
        \item[(T1)] Add a consequence $r \in \langle\langle R \rangle\rangle_{sq}$ to $R$.
        \item[(T2)] Delete a relator $r$ from $R$, where $r$ is a consequence of other relators in $R$.
        \item[(T3)] Add a new generator $x$ and a new relator $(x,r)$ to $A$ and $R$, respectively, where $r$ is an element of $F(A)$.
        \item[(T4)] Delete a generator $x$ and a relator $(x,r)$ from $A$ and $R$, respectively, where $(x,r)$ is a consequence of other relators and $x$ does not occur in other relators in $R$.
     \end{enumerate}
\end{proposition}

A symmetric quandle presentation is called a \textit{$(2m,n)$-presentation with inverses} (\textit{associated with $b_1, \dots, b_n$}) if it is
\[
    \left\langle x_1, \ldots, x_{2m}~
    \begin{array}{|c}
        b_i \cdot x_{1} = b_i \cdot x_{2}, ~ x_{2j-1}= \overline{x_{2j}} \quad (i=1,\ldots,n, ~j = 1,\ldots,m)
    \end{array}
    \right\rangle_{\mathrm{sq}}.
\]
We say that an $(2m,n)$-presentation with inverses satisfies the \textit{weak $\partial$-condition} if there exist signs $\varepsilon_1, \dots, \varepsilon_n \in \{1, -1\}$ such that
\[
    \prod_{i=1}^n b_i^{-1} \sigma_1^{\varepsilon_i} b_i \in K_{2m}.
\]

\begin{proposition}[\cite{Yasuda24}]\label{Proposition: knot symmetric quandle of plat form}
    Let $S$ be an adequate braided surface of degree $2m$ with a braid system $(\beta_1, \dots, \beta_n)$ such that $\beta_i = b_i^{-1} \sigma_1^{\,\varepsilon_i} b_i$ ($\varepsilon_i \in \{ 1, -1\}$).
    Then the knot symmetric quandle of the plat closure $\widetilde{S}$ has a $(2m,n)$-presentation with inverses associated with $b_1, \dots, b_n$.
\end{proposition}

% \begin{definition}[cf. \cite{Fenn-Rourke1992}]
%     Let $X$ be a symmetric quandle with a presentation $\langle A ~|~ R\rangle_{sq}$.
%     The following four moves are called \textit{Tietze moves} on a symmetric quandle presentation.
%     \begin{enumerate}
%        \item[(T1)] Add a consequence of $R$ to $R$.
%        \item[(T2)] Delete a consequence of other relators from $R$.
%        \item[(T3)] Introduce a new generator $a$ and a new relator $(a,r)$ to $A$ and $R$, respectively, where $r$ is an element of $F(A)$.
%        \item[(T4)] Delete a generator $a$ and a relator $(a,r)$ from $A$ and $R$, respectively, where $(a,r)$ is a consequence of other relators and $a$ does not occur in other relators in $R$.
%     \end{enumerate}
%  \end{definition}

%  Note that (T2) and (T4) are inverse moves of (T1) and (T3), respectively.
%  Fenn and Rourke \cite[Theorem~4.2]{Fenn-Rourke1992} proved Tietze move theorem for racks, and their proof can be extended for symmetric quandles.

%  \begin{proposition}[cf. \cite{Fenn-Rourke1992}]\label{Theorem: Tietze theorem for symmetric quandle}
%     Any two finite presentations of a symmetric quandle are related by a finite sequence of Tietze moves.
%  \end{proposition}

\subsection{Proof of Theorem~\ref{Main Theorem2}} % **
We prove Theorem~\ref{Main Theorem2} by rephrasing the proof of Theorem~\ref{Main theorem}.
Thus, we provide only an outline of the proof.

\noindent\textbf{(Proof of the only if part)}:
Let $(Q, \rho)$ be the knot symmetric quandle of a $(c, d)$-component genus $g$ surface-link $F$.
Then, there exists an adequate braided surface $S$ of degree $2m$ whose plat closure $\widetilde{S}$ is $F$.
By Proposition~\ref{Proposition: knot symmetric quandle of plat form}, $(Q, \rho)$ has a $(2m, n)$-presentation with inverses satisfying the weak $\partial$-condition.
Furthermore, $(Q, \rho)$ satisfies the condition (2) because $\chi(\widetilde{S}) = 2m -n = 2(c+d) - g$.

Let $F_0$ be a connected component of $F$.
For any $[(D, \alpha)], [(D', \alpha')] \in X(F)$ with meridional disks $D$ and $D'$ of $F_0$, either $[(D, \alpha)]$ or $[(-D, \alpha)]$ belong to the connected component of $X(F)$ containing $[(D', \alpha')]$.
Moreover, connected components of $X(F)$ containing $[(D, \alpha)]$ and $[(-D, \alpha)]$ are the same if and only if $F_0$ is non-orientable.
Thus, the number of connected component of $X(F)$ is $2c+d$.
Furthermore, the good involution of $X(F)$ maps $[(D, \alpha)]$ to $[(-D, \alpha)]$.
Hence, $(Q, \rho)$ satisfies the condition (3).

\noindent\textbf{(Proof of the if part)}:
Let $(Q, \rho)$ be a symmetric quandle satisfying the conditions (1) -- (3) of Theorem~\ref{Main Theorem2} for some $m,n \geq 0$.
Then, there exists an adequate braided surface $S$ of degree $2m$ such that $X(\widetilde{S})$ is isomorphic to $(Q, \rho)$.
By the condition (3), $\widetilde{S}$ is $(c, d)$-component.
Finally, the genus of $\widetilde{S}$ is $g$ since it holds that $\chi(\widetilde{S}) = 2m - n = 2(c+d) - g$.
% This completes the proof of Theorem~\ref{Main Theorem2}.
\qed

\section{Knot symmetric quandles and dihedral quandles}\label{Section: knot symmetric quandles and dihedral quandles}

In this section, we provide symmetric quandle presentations of dihedral quandles with arbitrary good involutions, and we show that dihedral quandles with arbitrary good involutions can be realized as the knot symmetric quandles of surface-links.
The following was provided by Taniguchi and Kamada in a private communication:
Let $\langle A \mid R \rangle_{\mathrm{sq}}$ be a symmetric quandle presentation of $(Q, \rho)$.
Denote $\overline{A} = \{\overline{x} \mid x \in A\}$, and assume that $\overline{A}$ is disjoint from $A$.
For an element $r = (x^w, y^u) \in R$, we denote $\overline{r} = (\overline{x}^w, \overline{y}^u)$, where $\overline{\overline{x}} = x$ for $x \in A$.
We define $\overline{R} = \{ \overline{r} \mid r \in R \}$, and $R_0 = \{(x^{\overline{y}}, x^{y^{-1}}), (\overline{x}^{\overline{y}},~ \overline{x}^{y^{-1}}) \mid x, y \in A \}$.
% We remark that $\overline{x} \in \overline{A}$ is not identified with $x^{-1} \in F(A)$ even though $\overline{x} \in \langle A \mid R \rangle_{\mathrm{sq}}$ is identified with $x^{-1} \in F(A)$ (Example~\ref{Example: the double of quandle}).

\begin{proposition}\label{Proposition: Quandle presentation from symmetric quandle presentation}
    In the situation above, $Q$ has a quandle presentation $\langle A, \overline{A} \mid R, \overline{R}, R_0 \rangle_{\mathrm{q}}$.
\end{proposition}

\begin{proof}
    Let $Q_0$ denote the quandle with $\langle A, \overline{A} \mid R, \overline{R}, R_0 \rangle_{\mathrm{q}}$.
    We define the involution $\rho_0: Q_0 \to Q_0$ by $\rho_0(x^w) = \overline{x}^w$ for $x^w \in Q_0$.
    Then we verify that $\rho_0$ is a good involution of $Q_0$ as follows:
    For any $x^w, y^u \in Q$ ($x, y \in A \cup \overline{A}$, $w, u \in F(A)$), we have
    \begin{align*}
        \rho_0(x^w * y^u) ~&=~ \overline{x}~^{wu^{-1} y u} ~=~ \rho_0(x^w) * y^u.\\
        x^w * \rho_0(y^u) ~&=~ x^{w u^{-1} \overline{y} u} ~\stackrel{(\sharp)}{=}~ x^{\overline{y}y w u^{-1} y^{-1} u} ~=~ x^{y^{-1}y w u^{-1} y^{-1} u} ~=~ x^w ~\overline{*}~ y^u.
    \end{align*}
    Here, the equality $(\sharp)$ follows from $z^{w\overline{y}} = z^{wx^{-1}\overline{y}yxy^{-1}}$ for each $x, y, z \in A$ and $w \in F(A)$, derived from the relation $x^{\overline{y}} = x^{y^{-1}}$ in $R_0$.

    Let $f_1: A \cup \overline{A} \hookrightarrow Q$ and $\iota_1: A \cup \overline{A} \hookrightarrow Q_0$ be the natural maps.
    By the universality of quandle presentations (cf. \cite[Lemma~8.6.3]{Kamada2017_book}), there exists a quandle homomorphism $F_1: Q_0 \to Q$ such that $f_1 = F_1 \circ \iota_1$.

    Similaly, let $f_2: A \hookrightarrow Q_0$ and $\iota_2: A \hookrightarrow Q$ be the natural maps.
    By the universality of symmetric quandle presentation (\cite[Lemma~3.8]{Yasuda24}), there exists a symmetric quandle homomorphism $F_2: (Q, \rho) \to (Q_0, \rho_0)$ such that $f_2 = F_2 \circ \iota_2$.
    Finally, $F_1$ and $F_2$ are inverse maps of each other, and thus $Q$ and $Q_0$ are isomorphic.
\end{proof}

\begin{lemma}[cf. \cite{Dhanwani-Raundal-Singh2023}]\label{Lemma: Presentations of dihedral quandles}
    The dihedral quandle has the following quandle presentation:
    \begin{align*}
        R_{2n+1} ~&=~ \left\langle x, y~
        \begin{array}{|c}
            x^{y^2} = x, ~ y^{x^2} = y, ~ x = y^{(x y)^{n}}
        \end{array}
        \right\rangle_{\mathrm{q}}.\\
        R_{2n} ~&=~ \left\langle x, y~
        \begin{array}{|c}
            x^{y^2} = x, ~ y^{x^2} = y, ~ x = x^{(y x)^{n}}, ~ y = y^{(x y)^{n}}
        \end{array}
        \right\rangle_{\mathrm{q}}.
    \end{align*}
\end{lemma}

Now, we provide symmetric quandle presentations of dihedral quandles for each good involution.

\begin{proposition}\label{Proposition: Presentations of dihedral quandles with anti}
    For $n \geq 1$, the dihedral quandle with the antipodal map has the following presentation:
    \begin{align*}
        (R_{4n+2}, \rho_A) ~&=~ \left\langle x, y~
        \begin{array}{|c}
            \overline{x} = y^{(x y)^{n}}, ~\overline{y} = x^{(yx)^{n}}
        \end{array}
        \right\rangle_{\mathrm{sq}},\\
        (R_{4n}, \rho_A) ~&=~ \left\langle x, y~
        \begin{array}{|c}
            \overline{x} = x^{(y x)^{n}}, ~ \overline{y} = y^{(x y)^{n}}
        \end{array}
        \right\rangle_{\mathrm{sq}}.\\
    \end{align*}
\end{proposition}

\begin{proof} % ^^^
    We show the case of $(R_{4n+2}, \rho_A)$.
    Let $(Q, \rho) = \left\langle x, y~
    \begin{array}{|c}
        \overline{x} = y^{(x y)^{n}}, ~\overline{y} = x^{(yx)^{n}}
    \end{array}
    \right\rangle_{\mathrm{sq}}$ denote the symmetric quandle.
    By direct calculation, we obtain the equality $x*y = x*(y^{(xy)^{n}} x^{(yx)^{n}y})$ without any additional relations:
    \begin{align*}
        x*(y^{(xy)^n} x^{(yx)^{n}y}) ~&=~ x*((y^{-1}x^{-1})^n ~y (xy)^n ~y^{-1} (x^{-1} y^{-1})^{n} x (yx)^{n} y)\\
        ~&=~ x*((y^{-1}x^{-1})^n ~(yx)^n y ~y^{-1} (x^{-1} y^{-1})^{n} (xy)^{n+1})\\
        ~&=~ x*(xy) ~=~ x*y.
    \end{align*}
    Since the relation $\overline{y} = x^{(yx)^{n}}$ derives the equality $\overline{y} = x^{(yx)^{n}y}$, we obtain a new relation $x*y = x* \overline{y}$ in the presentation.
    Similarly, we obtain $y*x = y*\overline{x}$.
    Hence, $(Q, \rho)$ has the following presentation
    \begin{align}\label{align: prese}
        (Q, \rho) ~&=~ \left\langle x, y~
        \begin{array}{|c}
            x = x^{y^2}, ~ y = y^{x^2},~
            \overline{x} = y^{(x y)^{n}}, ~\overline{y} = x^{(yx)^{n}}
        \end{array}
        \right\rangle_{\mathrm{sq}}.
    \end{align}
    By Proposition~\ref{Proposition: Quandle presentation from symmetric quandle presentation}, $Q$ has a quandle presentation with four generators $x, y, \overline{x}$, $\overline{y}$, and the following 16 relations:
    \begin{gather*}
        x = x^{y^2}, \quad y = y^{x^2}, \quad \overline{x} = y^{(x y)^{n}}, \quad \overline{y} = x^{(yx)^{n}}, \qquad
        \overline{x} = \overline{x}^{y^2}, \quad \overline{y} = \overline{y}^{x^2}, \quad x = \overline{y}^{(x y)^{n}}, \quad y = \overline{x}^{(yx)^{n}},\\
        x^{\overline{x}} = x^{x^{-1}}, \quad x^{\overline{y}} = x^{y^{-1}}, \quad y^{\overline{x}} = y^{x^{-1}}, \quad y^{\overline{y}} = y^{y^{-1}}, \qquad
        \overline{x}^{\overline{x}} = \overline{x}^{x^{-1}}, \quad \overline{x}^{\overline{y}} = \overline{x}^{y^{-1}}, \quad \overline{y}^{\overline{x}} = \overleftarrow{y}^{x^{-1}}, \quad \overline{y}^{\overline{y}} = \overline{y}^{y^{-1}},
    \end{gather*}
    where the first four relations are came from the presentation of $(Q,\rho)$, the next four relations are came from $\overline{R}$ of Proposition~\ref{Proposition: Quandle presentation from symmetric quandle presentation}, and the relations in the second row are came from $R_0$ of Proposition~\ref{Proposition: Quandle presentation from symmetric quandle presentation}.

    We now remove generators $\overline{x}, \overline{y}$ with the relations $\overline{x} = y^{(x y)^{n}}$ and $\overline{y} = x^{(y x)^{n}}$ from the presentation of $Q$, resulting in the quandle $Q$ being generated by $x, y$ with the following 14 relations:
    \begin{gather*}
        x = x^{y^2}, \quad y = y^{x^2},\qquad
        y^{(x y)^{n} y^2} = y^{(x y)^{n}}, \quad x^{(x y)^{n} x^2} = x^{(x y)^{n}}, \quad x = x^{(yx)^{n} (x y)^{n}}, \quad y = y^{(x y)^{n} (yx)^{n}},\\
        x^{y^{(x y)^{n}}} = x^{x^{-1}}, \quad x^{x^{(y x)^{n}}} = x^{y^{-1}}, \quad y^{y^{(x y)^{n}}} = y^{x^{-1}}, \quad y^{x^{(y x)^{n}}} = y^{y^{-1}},\\
        (y^{(x y)^{n}})^{y^{(x y)^{n}}} = (y^{(x y)^{n}})^{x^{-1}}, \quad (y^{(x y)^{n}})^{x^{(y x)^{n}}} = (y^{(x y)^{n}})^{y^{-1}}, \quad (x^{(y x)^{n}})^{y^{(x y)^{n}}} = (x^{(y x)^{n}})^{x^{-1}}, \quad (x^{(y x)^{n}})^{x^{(y x)^{n}}} = (x^{(y x)^{n}})^{y^{-1}}.
    \end{gather*}
    Two relations $x^{y^2} = x$ and $y^{x^2} = y$ induce the following equalities for any $z \in \{x, y\}$ and $w, u \in F(\{x, y\})$:
    \begin{align*}
        z^{wxu} = z^{wx^{-1}u}, \quad z^{wyu} = z^{wy^{-1}u}.
    \end{align*}
    Then, the above 14 relations are simplified as follows:
    \begin{gather*}
        x = x^{y^2}, \quad y = y^{x^2},\qquad
        y^{(x y)^{n}} = y^{(x y)^{n}}, \quad x^{(x y)^{n}} = x^{(x y)^{n}}, \quad x = x, \quad y = y,\\
        x^{(yx)^{2n+1}} = x, \quad x^{(yx)^{2n+1}} = x, \quad y^{(xy)^{2n+1}} = y, \quad y^{(xy)^{2n+1}} = y^,\\
        y^{(xy)^{2n+1}} = y, \quad y^{(xy)^{2n+1}} = y, \quad x^{(yx)^{2n+1}} = y, \quad x^{(y x)^{2n+1}} = y.
    \end{gather*}

    Therefore, $Q$ is the dihedral quandle $R_{2n+1}$.
    Since the images of $x$ and $y$ under $\rho$ are distinct from themselves, $\rho$ is the antipodal map $\rho_A$.
    The case of $(R_{4n}, \rho_A)$ can be shown by a similar way, so we leave the proof to the reader.
\end{proof}

The proofs of the following two propositions are similar to the proof of Proposition~\ref{Proposition: Presentations of dihedral quandles with anti}, so we omit them.

\begin{proposition}\label{Proposition: Presentations of dihedral quandles with id}
    The dihedral quandle with the idetity map has the following presentations:
    \begin{align*}
        (R_{2n+1}, \id) ~&=~ \left\langle x, y~
        \begin{array}{|c}
            \overline{x} = x, ~ \overline{y} = y,~
            x = y^{(x y)^{n}}
        \end{array}
        \right\rangle_{\mathrm{sq}},\\
        (R_{2n}, \id) ~&=~ \left\langle x, y~
        \begin{array}{|c}
            \overline{x} = x, ~ \overline{y} = y,~
            x = x^{(y x)^{n}}, ~ y = y^{(x y)^{n}}
        \end{array}
        \right\rangle_{\mathrm{sq}}.\\
    \end{align*}
\end{proposition}

\begin{proposition}\label{Proposition: Presentations of dihedral quandles with half anti}
    The dihedral quandle with the half-antipodal map has the following presentation:
    \begin{align*}
        (R_{4n}, \rho_{HA}) ~=~ \left\langle x, y~
        \begin{array}{|c}
            \overline{x} = x, ~ \overline{x} = x^{(y x)^n}, ~ \overline{y} = y^{(x y)^n}
        \end{array}
        \right\rangle_{\mathrm{sq}}.
    \end{align*}
\end{proposition}

% \begin{theorem}\label{Theorem: dihedral quandle and knot symmetric quandle}
%     For any $n \geq 1$ and any good involutions $\rho$ of $R_n$, there exists a surface-link such that its knot symmetric quandle is isomorphic to $(R_n, \rho)$.
% \end{theorem}

\subsection{Proof of Thoerem~\ref{Theorem: dihedral quandle and knot symmetric quandle}}
For each $n \geq 1$ and each good involution $\rho$ of $R_n$, we construct a surface-link $F$ with the knot symmetric quandle $X(F)$ isomorphic to $(R_n, \rho)$.

\noindent (\textbf{Case 1}: $n \equiv 2~(\mathrm{mod}~4)$, $\rho = \rho_{A}$: the antipodal map)
We denote $k = n/2$ and take an integer $m$ such that $\gcd(k,m) = 1$.
Let $F = \tau^2(K(k, m))$ denote the 2-twist spin of a 2-bridge knot $K(k, m)$ (\cite{Zeeman1965}).
It was shown by Inoue \cite{Inoue2019} that the knot quandle $Q(F)$ is isomorphic to the dihedral quandle $R_{k}$.
% Furthermore,it was shown by Oshiro \cite{Oshiro2020-preprint} that the double of $R_n$ is the dihedral quandle $(R_{2n}, \rho_A)$ with the antipodal map.
Hence, the knot symmetric quandle $X(F)$ is the double of $R_{k}$, which is isomorphic to $(R_{2k}, \rho_A) = (R_{n}, \rho_A)$ (see Example~\ref{Example: the double of quandle}).

\noindent (\textbf{Case 2}: $n \equiv 0~(\mathrm{mod}~4)$, $\rho = \rho_{A}$)
We construct a surface-link $F_n$ as follows:
Let $b = (\sigma_2^{-1})^{k}$ be a 4-braid for $k = n/2$.
We put $b_1 = b$ and $b_2 = (\sigma_2\sigma_1\sigma_3\sigma_2)^{-1} \, b$.
Then, $(b_1^{-1}\sigma_1 b_1) (b_2^{-1} \sigma_1^{-1} b_2)$ is an adequate 4-braid.
Hence, we have a symmetric quandle $X$ with a $(4,2)$-presentation with inverses associated with $b_1$, $b_2$ satisfying the weak $\partial$-condition:
\begin{align*}
    \left\langle x_1, x_2, x_3, x_4~
    \begin{array}{|c}
        x_1 = x_2^{(x_3 ~ x_2)^k}, ~ x_3^{(x_2 x_3)^k} = x_4, ~
        x_1 = \overline{x_2}, ~ x_3 = \overline{x_4}
    \end{array}
    \right\rangle_{\mathrm{sq}}.
\end{align*}
We remove generetors $x_1$ and $x_4$ with relations $x_1 = \overline{x_2}$ and $x_3 = \overline{x_4}$ from the presentation so that we have
\begin{align*}
    \left\langle x_2, x_3~
    \begin{array}{|c}
        \overline{x_2} = x_2^{(x_3 x_2)^k}, ~ \overline{x_3} = x_3^{(x_2 x_3)^k}
    \end{array}
    \right\rangle_{\mathrm{sq}}.
\end{align*}
By Proposition~\ref{Proposition: Presentations of dihedral quandles with anti}, $X$ is the dihedral quandle $R_{2k} = R_n$ with the antipodal map.
By Theorem~\ref{Main Theorem2}, $X$ is the knot symmetric quandle of a surface-link $F_n$ consisting of two projective planes.

\noindent(\textbf{Case 3}: $n \equiv 1, 3~(\mathrm{mod}~4)$, $\rho = \id$)
Let $P_0$ be a standard projective plane in $\R^4$ (\cite{Kamada1989}).
We define $F$ as a connected sum of $\tau^2(K(k,m))$ and $P_0$.
In terms of the presentation of a symmetric quandle, taking the connected sum of $P_0$ is equivalent to adding relations $x = \overline{x}$, where $x$ represents the meridional disk of $P_0$.
Since $F$ is connected, $x = \overline{x}$ induces the relation $y = \overline{y}$ so that $X(F)$ has the presentation of $(R_{2n+1}, \id)$ in Proposition~\ref{Proposition: Presentations of dihedral quandles with id}.

\noindent(\textbf{Case 4}: $n \equiv 0, 2~(\mathrm{mod}~4)$, $\rho = \id$)
Let $F_n$ denote the surface-link constructed in (Case 2).
We obtain the surface-link $F$ by taking the connected sum of $P_0$ with each component of $F_n$, resulting in $F$ consisting of two Klein bottles.
Then, a presentation of $X(F)$ is obtained from the presentation of $(R_{4n}, \rho_A)$ in Proposition~\ref{Proposition: Presentations of dihedral quandles with anti} by adding two relations $x = \overline{x}$ and $y = \overline{y}$.
Hence $X(F)$ has the presentation of $(R_{2n}, \id)$ in Proposition~\ref{Proposition: Presentations of dihedral quandles with id}.

\noindent(\textbf{Case 5}: $n \equiv 0~(\mathrm{mod}~4)$, $\rho = \rho_{HA}$: the half-antipodal map)
Let $F$ be a connected sum of $F_n$ and a single $P_0$, resulting in $F$ consisting of a projective plane and a Klein bottle.
Then a presentation of $X(F)$ is obtained from the presentation of $(R_{4n}, \rho_A)$ in Proposition~\ref{Proposition: Presentations of dihedral quandles with anti} by adding the relation $x = \overline{x}$.
By Proposition~\ref{Proposition: Presentations of dihedral quandles with half anti}, $X(F)$ is isomorphic to $(R_{4n}, \rho_{HA})$.
\qed

As a consequence of Theorem~\ref{Theorem: dihedral quandle and knot symmetric quandle}, the knot symmetric quandle is a stronger invariant than the full knot quandle for distinguishing surface-links:

\begin{corollary}\label{Corollary: full knot quandle vs knot symmetric quandle}
    There exist surface-links that share the same full knot quandle but have distinct knot symmetric quandles.
\end{corollary}

\section{Remarks on $P^2$-irreduciblity of non-orientable surface-links}\label{Section: Final remarks}

% He also showed the knot groups of them are isomorphic to $Q_1$ and $Q_2$, respectively.
% A \textit{$P^2$-knot} is a surface-knot homeomorphic to the real projective plane, and a \textit{$P^2$-link} is a surface-link consisting of $P^2$-knots.
A \textit{$P^2$-knot} is a surface-knot homeomorphic to the projective plane $\RP^2$, and a \textit{$P^2$-link} is a disjoint union of $P^2$-knots.
A surface-link is called \textit{$P^2$-reducible} if it is a connected sum of a standard $P^2$-knot (\cite{Kamada1989}) and some surface-link, otherwise it is called \textit{$P^2$-irreducible}.
% According to $P^2$-reduciblity, the following conjecture is
The following is one of the fundamental conjectures for $P^2$-knots.

\begin{conjecture}[Kinoshita's problem \cite{Kinoshita1961}]
    Every $P^2$-knot is $P^2$-reducible, that is, every $P^2$-knot is a connected sum of a standard $P^2$-knot and a surface-knot homeomorphic to the 2-sphere $S^2$.
\end{conjecture}

Using the knot symmetric quandle, we derive a necessary condition for the $P^2$-reduciblity of surface-knots.

\begin{proposition}\label{Proposition: P2-reduciblity for surace-knots}
    For a $P^2$-reducible surface-knot $F$, the good involution of $X(F)$ is the identity map.
    In particular, $X(F)$ is a kei.
\end{proposition}

\begin{proof}
    Let $F$ be a $P^2$-reducible surface-knot, meaning that $F$ decomposes into a surface-knot and a trivial $P^2$-knot $P_0$.
    Let $x \in X(F)$ be an element representing a meridional disk of $P_0$.
    Then the image $\rho(x)$ of $x$ under the good involution $\rho$ of $X(F)$ is the same as $x$.
    Since $F$ is a surface-knot with a $(0,1)$-component, $X(F)$ is a connected quandle by Theorem~\ref{Main Theorem2}.
    Hence, for each $y \in X(F)$, there exists a word $w \in \mathrm{F}(X(F))$ such that $y = x*w$, where $\mathrm{F}(X(F))$ is the free group on $X(F)$.
    Therefore, we have $\rho(y) = \rho(x*w) = \rho(x)*w = x*w = y$, which implies that $\rho = \id$.
    It was shown by Kamada-Oshiro \cite{Kamada-Oshiro2010} that a quandle $Q$ is a kei if and only if the identity map is a good involution of $Q$.
\end{proof}

We remark that this Proposition does not hold for surface-links.
For instance, the surface-link constructed in (Case 5) of the proof of Theorem~\ref{Theorem: dihedral quandle and knot symmetric quandle} is $P^2$-reducible even though the good involution is the half-antipodal map.

Yoshikawa \cite{Yoshikawa1994} constructed two $P^2$-irreducible $P^2$-links, denoted by $8_1^{-1,-1}$ and $10_1^{-1,-1}$ in the list of \cite{Yoshikawa1994}.
In the proof of Theorem~\ref{Theorem: dihedral quandle and knot symmetric quandle}, we construct an infinite family of 2-component $P^2$-links $F_n$ $(n \geq 1)$.
We notice that $F_1$ and $F_2$ are $8_1^{-1,-1}$ and $10_1^{-1,-1}$, respectively.
It follows from a direct computation that the knot group of $F_n$ is the generalized quaternion group $Q_{n}$ of order $8n$ whose group presentation is
\begin{align*}
    \left\langle a, b~
    \begin{array}{|c}
        a^{4k} = 1, b^2 = a^k, b^{-1}ab = a^{-1}
    \end{array}
    \right\rangle,
\end{align*}
where meridional loops of the components of $F_n$ are represented by $b$ and $b^{-1}a$, respectively.
The orders of $b$ and $b^{-1}a$ are four.
It implies that $F_n$ is $P^2$-irreducible, and we have the following result:

\begin{theorem}\label{Theorem: P2-irreducible P2-links}
    There exists an infinite family of 2-component $P^2$-irreducible $P^2$-links.
\end{theorem}

% In \cite{Yoshikawa1994}, Yoshikawa provided two examples of 2-component $P^2$-irreducible $P^2$-links, denoted by $8_1^{-1,-1}$ and $10_1^{-1,-1}$.
% By constructing a plat form presentation of these $P^2$-links, we see that $8_1^{-1,-1}$ and $10_1^{-1,-1}$ are equivalent to $F_1$ and $F_2$, respectively.

\section*{Acknowledgment}
The author would like to thank Taizo Kanenobu and Yuta Taniguchi for their helpful advice on the computation of the knot groups and the knot symmetric quandles of surface-links.
He also would like to thank Seiichi Kamada for helpful comments on this work.
This work was supported by JSPS KAKENHI Grant Number 22J20494.

\bibliographystyle{plain}
\bibliography{reference.bib}
\end{document}